\documentclass[10pt]{amsart}
\usepackage[utf8x]{inputenc}

\usepackage{dsfont}
\usepackage{amsmath}
\usepackage{amssymb}
\usepackage{graphicx}
\usepackage{amssymb}
\usepackage{amsfonts}
\usepackage{soul}
\usepackage{mathpazo}
\usepackage{color}
\usepackage{yfonts}
\usepackage{paralist}
\usepackage{stmaryrd}
\usepackage{amsxtra}
\usepackage{latexsym,mathrsfs}

\usepackage{mathabx}

\usepackage{epsfig, enumerate}
\usepackage{color}
\usepackage{hyperref}

\newtheorem{theorem}{Theorem}
\newtheorem{lemma}[theorem]{Lemma}
\newtheorem{coro}[theorem]{Corollary}

\newtheorem{conj}[theorem]{Conjecture}

\newtheorem{prop}[theorem]{Proposition}
\newtheorem{Remark}[theorem]{Remark}

\numberwithin{theorem}{section}
\numberwithin{equation}{section}

\title[Tur\'an Number and feedback vertex numbers]{On Tur\'an Number of Graphs with Small Minimum Feedback Vertex Numbers}

\author{Xiao-Chuan Liu}
\address[Liu]{Departamento de Matemática,
 Universidade Federal de Pernambuco,
	Avenida Jornalista Aníbal Fernandes, Cidade Universitária, Recife, Brasil}
\email{xiaochuan.liu@ufpe.br}

\author{Xu Yang}
\address[Yang]{Instituto de Computação, Universidade Federal de Alagoas,
	Av. Lourival Melo Mota, S/N, Maceió, Brasil} 
\email{yang@ic.ufal.br}

\begin{document}
\maketitle{}
\begin{abstract} 
Given a graph $H$, the minimum feedback vertex number of $H$ is the minimum number of vertices whose removal results in an acyclic graph. In this paper, we investigate Turán-type extremal problems for bipartite graphs in terms of their feedback vertex number. Our first result concerns bipartite graphs $H$ with minimum feedback vertex number one. Such graphs can be 
obtained from a forest by identifying a specified collection of leaves into a single vertex.
For these graphs, we show that $\text{ex}(n, H)$ is upper bounded by $O(n^{1+1/k^\ast})$,  
where $2k^\ast$ is the length of the shortest cycle contained in 
$H$.  

In addition, we consider a family of bipartite graphs with minimum feedback vertex number three. Let $E_{k,t}$ be the graph obtained from the theta graph $\theta_{k,t}$ by joining a new vertex $x$ to one side of the bipartition and another vertex $y$ to the other. Let $E^+_{k,t}$ denote the graph obtained by adding the edge $xy$ to $E_{k,t}$. We prove that for any $k\geq 2$ and sufficiently large $t$, 
$\text{ex}(n, E^+_{k,t})= \Theta(n^{\frac{3k-1}{2k-1}}).$

\end{abstract}

\section{Introduction}
\label{sec:introduction}
The {\it $\text{extremal number}$}, or  {\it Tur\'an number}, of a graph $H$, denoted by $\text{ex}(n, H)$, is defined as
the maximum number of edges in an $H$-free graph on $n$ vertices. The celebrated Erd\H os-Stone-Simonovits theorem relates
$\text{ex}(n,H)$ 
to the chromatic number $\chi(H)$, and 
states that 
\begin{align}
	\text{ex}(n, H) = \bigg(1-\frac{1}{\chi(H)-1}+o(1)\bigg)\binom{n}{2}.
\end{align}    
When $H$ is bipartite, the Erd\H os-Stone-Simonovits theorem only provides a crude order of magnitude and does not yield accurate estimates. 
A classic example is the family of even cycles $C_{2k}$, which are among the most extensively studied bipartite graphs in extremal graph theory.  
It has long been known that 
$\text{ex}(n,C_{2k})=O(n^{1+1/k})$ for all $k\geq 2$, 
although
matching lower bounds are only known for $k=2,3,5$ (see~\cite{wenger1991extremal} by Wenger).
Other interesting and somewhat related bipartite graphs have also been studied intensively. 
Researchers have studied the following theta graphs. 
The {\it theta graph} $\theta_{k,\ell}$, for $k,\ell\geq 2$, 
consists of two vertices connected by $\ell$ internally disjoint paths of length $k$. 
Faudree and Simonovits~\cite{Faudree1983} proved that $\text{ex}(n,\theta_{k,\ell})=O(n^{1+1/k})$.
The corresponding lower bounds were later obtained by Conlon~\cite{Conlon2019graphs} 
for all $\ell$ sufficiently large, depending on $k$.  

Very recently, Liu and Yang~\cite{doi:10.1137/21M1408439} studied the
\textit{generalized theta graph} $\Theta_{k_1,\dots,k_\ell}$, which
consists of two vertices joined by $\ell$ internally disjoint paths of
lengths $k_1,\dots,k_\ell$, where all the $k_i$ have the same parity.
Define
\[
k^\ast=\frac12\min_{1\leq i<j\leq \ell}(k_i+k_j),
\]
so that $2k^\ast$ is the length of the shortest cycle in the graph.
They proved that
\[
\operatorname{ex}(n,\Theta_{k_1,\dots,k_\ell})
=
O(n^{1+1/k^\ast}).
\]
Thus, for generalized theta graphs, the extremal upper-bound exponent
is governed by the length of the shortest cycle.

In this paper, we extend this phenomenon from generalized theta graphs to all connected bipartite graphs with minimum feedback vertex number one. Recall that a \textit{feedback vertex set} of a graph is a set of vertices whose removal results in an acyclic graph, and the \textit{minimum feedback vertex number} (FVN) is the minimum size of such a set. Connected graphs with minimum FVN equal to $1$ can be obtained by identifying all the leaves of a suitable forest to a single vertex.

This represents a substantial increase in structural complexity. Generalized theta graphs consist of internally disjoint paths joining two fixed vertices, whereas a general graph in our class may contain extensive branching and need not admit any comparable path representation. Thus, Theorem~\ref{main} extends a result for a rigid family described by path lengths to an entire structural class of bipartite graphs.

\begin{theorem}\label{main} 
Let $H$ be a connected bipartite graph with  minimum feedback vertex number $1$. 
Suppose $2k^\ast$ is the length of the shortest cycle contained in $H$.
Then $\text{ex}(n, H)=O(n^{1+1/k^\ast})$. 
\end{theorem}

We remark that matching lower bounds are generally the main difficulty in this
setting. Two sharp constructions illustrate the situation at exponent $5/4$.
Verstraete and Williford~\cite{verstraete2019graphs} constructed
$\theta_{4,3}$-free graphs with $\Omega(n^{5/4})$ edges, where
$\theta_{4,3}$ denotes the graph consisting of three internally disjoint paths
of length four with the same endpoints. This graph is the most immediate theta
extension of $C_8$, obtained by adding a third path of length four between the
same pair of endpoints. Together with the corresponding upper bound, their
construction yields
\begin{equation*}
    \operatorname{ex}(n,\theta_{4,3})=\Theta(n^{5/4}).
\end{equation*}
Liu and Yang~\cite{doi:10.1137/21M1408439} considered the generalized theta
graph $\Theta_{3,5,5}$, consisting of three internally disjoint paths of
lengths $3$, $5$, and $5$ with the same endpoints. This is a natural
non-uniform generalized theta graph whose shortest cycles have length eight.
They provided a matching lower-bound construction and proved that
\begin{equation*}
    \operatorname{ex}(n,\Theta_{3,5,5})=\Theta(n^{5/4}).
\end{equation*}
Thus, the exponent determined by the shortest cycle is sharp not only for the
uniform theta graph $\theta_{4,3}$, but also for a genuinely non-uniform
generalized theta graph.

Theorem~\ref{main} extends this shortest-cycle upper-bound principle from
generalized theta graphs to the entire class of bipartite graphs with minimum
feedback vertex number one. It therefore gives rise to many further natural
families for which matching lower-bound constructions remain unknown. One of
the simplest examples beyond generalized theta graphs is the graph $W_t$
obtained from $t$ copies of $C_8$ by identifying one vertex from each copy
into a single common vertex. At first sight, determining the Tur\'an number of
$W_t$ may appear to be a relaxation of the long-standing conjecture
\begin{equation*}
    \operatorname{ex}(n,C_8)=\Theta(n^{5/4}).
\end{equation*}
However, for every fixed $t$, determining the correct order of magnitude of
$\operatorname{ex}(n,W_t)$ is equivalent, up to a constant factor depending
only on $t$, to determining the correct order of magnitude of
$\operatorname{ex}(n,C_8)$. We prove this in
Proposition~\ref{prop:C8_bouquet} in Appendix~A.

We now consider a family of graphs whose minimum FVN is equal to three
whenever $k\geq 3$. Let $E_{k,t}$ be constructed as follows. Start with
a copy of the theta graph $\theta_{k,t}$, viewed as a bipartite graph
with bipartition $A\cup B$. Introduce two new vertices, $x$ and $y$,
and join $x$ to every vertex in $A$ and $y$ to every vertex in $B$.
The resulting graph is denoted by $E_{k,t}$. Furthermore, let
$E^+_{k,t}$ be obtained from $E_{k,t}$ by adding the edge $xy$.

For $k\geq 3$, the graph $E^+_{k,t}$ has minimum FVN equal to three.
Indeed, deleting $x$, $y$, and one endpoint of the original theta graph
results in an acyclic graph, whereas no set of two vertices meets every
cycle. When $k=2$, however, we have
$E^+_{2,t}=K_{3,t+1}$, whose minimum FVN is two, since deleting two
vertices from the part of size three leaves a star.

Our second main result establishes matching upper and lower bounds,
up to a constant factor, for the Turan number of $E^+_{k,t}$.

\begin{theorem}\label{new_exponents}
For any \( k \geq 2 \) and \( t \geq t_0(k) \), where \( t_0(k) \) is a sufficiently large constant depending on \( k \), we have:
\begin{equation}
\mathrm{ex}(n, E^+_{k,t}) = \Theta\left(n^{\frac{3k-1}{2k-1}}\right).
\end{equation}
\end{theorem}

We remark that \( 3/2 \) is a limit point of all the Turán exponents in the above theorem as \( k \to \infty \). Although the theorem is stated under the assumption that \( t \geq t_0(k) \) for some sufficiently large constant \( t_0(k) \), the upper bound it provides actually holds for all \( k \geq 2 \) and \( t \geq 2 \), as will be shown in Proposition~\ref{upper_generalized_cube}. When \( k = 2 \) and \( t = 2 \), the graph \( E_{2,2} \), also seen as $K_{3,3}$ minus one edge, 
is known to have Turán number \( \Theta(n^{3/2}) \), as shown in~\cite{furedi2001ramsey}. 
On the other hand, the graph $E_{2,3}$ contains the complete bipartite graph $K_{3,3}$ as a subgraph, and therefore, by Proposition~\ref{upper_generalized_cube}, 
we have $\text{ex}(n,E_{2,3})=\Theta(n^{5/3})$, as expected. 
When \( k = 3 \), we have \( E_{3,2} = Q_8 \), the 3-dimensional cube graph. In this case, our theorem recovers the upper bound \(\mathrm{ex}(n, Q_8) = O(n^{8/5})\), which is the best known to date. Determining the exact Turán exponent for this graph remains one of the major open problems in extremal graph theory, as the current best lower bound is \(\Omega(n^{3/2})\).
Following Theorem~\ref{new_exponents}
(cf. Proposition~\ref{upper_generalized_cube}), 
it is natural to ask: what is the least $t\geq 2$ such that $\text{ex}(n, E^+_{3,t})=\Theta(n^{8/5})$? 
Clearly, similar questions can be posed for each value of $k$. 

Section~2 is devoted to the proof of Theorem~\ref{main}. We argue by contradiction: assuming that the host graph $G$ has sufficiently many edges but contains no copy of $H$, we construct an embedding of $H$ into $G$. The main difficulty is that a general connected bipartite graph with minimum feedback vertex number one may have substantial branching and need not admit the rigid path structure of a theta graph or a generalized theta graph.

To overcome this difficulty, we introduce a structural decomposition of $H$ into a sequence of trees that can be embedded incrementally. We then develop a flexible path-family embedding lemma that allows the embedding process to be carried out inside a prescribed family of layered paths. Our argument builds on the path-counting method of Bukh and Tait~\cite{Bukh_theta}. We note, however, that the proof of their key bad-set estimate contains a gap: the path families appearing in the decomposition need not be disjoint, so the asserted counting identity does not follow. We repair this issue by introducing a residual-family decomposition, which assigns each path according to the last layer at which an irregular pair occurs. Together, the structural decomposition, the path-family embedding lemma, and the residual-family argument yield a rigorous proof of the required bad-set estimate and allow us to extend the method from theta-type graphs to all connected bipartite graphs with minimum feedback vertex number one.

The proof of Theorem~\ref{new_exponents} is presented in Section 3. The upper bound is again established via an embedding scheme, while the lower bound follows from Bukh’s random polynomial method. 

\section{Proof of Theorem~\ref{main}}\label{sec:proof}
We begin by introducing three preliminary lemmas, all of which are standard. Note that these lemmas will also be used in the next section.
\subsection{Preliminary Lemmas}
We write $G=(V(G),E(G))$ to denote that a graph $G$ has vertex set $V(G)$ and edge set $E(G)$. We use $v(G)$ and $e(G)$ to denote the number of
vertices and edges of $G$, respectively. 
\begin{lemma}\label{embedding_tree} 
	Let $T$ be a tree on $p$ vertices, and let $G$ be
	a bipartite graph in which every vertex has degree at least 
	$p+1$. Then $T$ can be embedded in $G$. 
	Moreover, for any prescribed vertex $w\in V(T)$ and any vertex $u\in V(G)$, 
	the embedding can be chosen so that $w$ is mapped to $u$. 
\end{lemma}
\begin{proof} The proof is elementary. One simply embeds the vertices of $T$ into $G$ one by one in a greedy manner. 
	To establish the final claim, we choose $w$ as the first vertex to embed and map it into $u$, then proceed with the greedy embedding process. 
\end{proof}

A graph $G$ is called \textit{$\Delta$-almost regular}
if every vertex $v$ has degree in some interval of the form
$[d, \Delta d]$.
A bipartite graph with bipartition $X\cup Y$ is called
\textit{balanced}  if $1/2\le |X|/|Y|\le 2$.
The following useful lemma will allow us to assume that our host graph is balanced and $\Delta$-almost regular.

\begin{lemma}[Lemma 2.3 of~\cite{conlon2021extremal}]\label{Kalmost_regular}
For any positive constant $\alpha<1$, there exists $n_0$ such that if $n\geq n_0$, $C\geq 1$ and $G$ is an $n$-vertex graph with at least $Cn^{1+\alpha}$ edges, then $G$ has a $\Delta$-almost regular balanced bipartite subgraph $G'$ with $m$ vertices such that $m\geq n^{\frac{\alpha(1-\alpha)}{2(1+\alpha)}}$, $|E(G')|\geq \frac{C}{10}m^{1+\alpha}$ and $\Delta=60\cdot 2^{1+1/\alpha^2}$.
\end{lemma}


The following lemma ensures that we can extract a subgraph of sufficiently high density where the minimum degree of each vertex on both sides of the bipartition is controlled. The proof is folklore.

\begin{lemma}\label{reduction}
Let $G$ be a bipartite graph with the vertex bipartition $X\cup Y$. Let  $d_1=e(G)/|X|$ and $d_2=e(G)/|Y|$ denote the average degrees of vertices in $X$ and $Y$, respectively. Then, by deleting at most half of the edges, we can obtain a subgraph $G'$ in which every vertex $x\in X$ has degree at least $d_1/4$, and every vertex $y\in Y$ has degree at least $d_2/4$.
\end{lemma}

\subsection{FVN 1 Graph Decomposition}
\label{subsec:key_embedding}
Consider a graph $H$ whose minimum FVN is equal to $1$. In this subsection we 
provide a decomposition of $H$, which will be used to greedily embed
 the entire graph. Informally speaking, it gives a convenient way to ``grow'' the graph $H$ starting from the special vertex $u$, whose removal makes the graph acyclic.

Before stating the proposition, we introduce some notation.  
First, suppose $u\in H$ is a designated vertex whose removal makes the graph $H$ acyclic. Fix an integer $2\leq k\leq k^\ast$. Define the subgraph $S_{k-1} \subseteq H$ to be the subgraph induced by all vertices at distance at most $k-1$ from $u$.
Note that $S_{k-1}$ is a tree. Let $F_{k-1}$ denote the set of vertices at distance exactly $k-1$ from $u$,  and define $S_{k-1}^-=S_{k-1}\setminus F_{k-1}$. Let $T$ be the union of the components of $H\backslash\{u\}$ that do not intersect $F_{k-1}$. Note that $T$ can be empty. 
Let $$T_{\mathrm{rest}}=H[V(T)\cup \{u\}].$$ 
Since $H$ is connected, $T_{\mathrm{rest}}$ is connected. Moreover, each component of $T$ is joined to $u$ by exactly one edge; otherwise, two such edges together with the unique path between their other endpoints would form a cycle of length less than $2k^\ast$. Therefore, $T_{\mathrm{rest}}$ is a tree, possibly the trivial tree consisting only of $u$.

Define $R_{k-1}:=H \backslash S_{k-1}$, where $H \backslash S_{k-1}$ denotes the graph obtained from $H$ by deleting all vertices of $S_{k-1}$ (together with their incident edges). 
For each vertex $w\in F_{k-1}$, let $P_{wu}$ denote the unique shortest path from $w$ to $u$ in $H$.
See Figure~\ref{figure} for an illustration of the notation.

\begin{figure}[htb]
	\centering
	\includegraphics[width=0.6\textwidth]{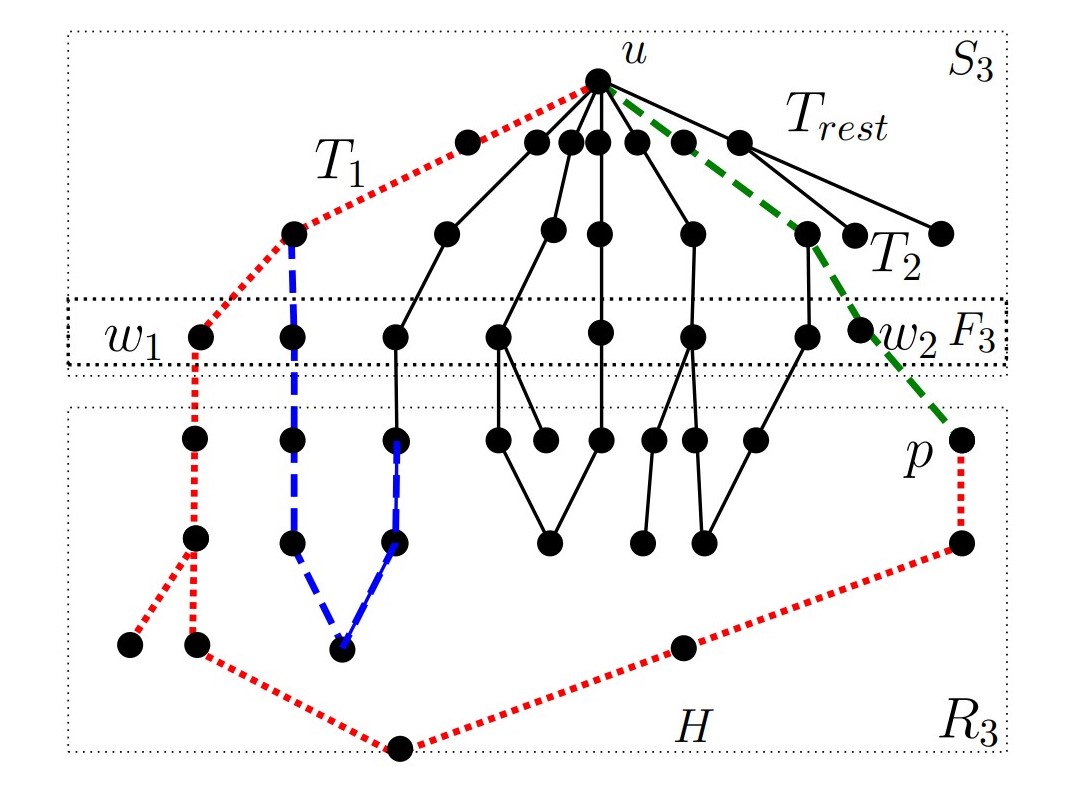}
	\caption{An example of the notation and the proof of Proposition~\ref{structure_of_F}. 
The red dotted tree is $T_1$. 
The blue dashed tree illustrates
one possible choice for the second tree.
The green dashed tree illustrates an alternative choice for the second tree, denoted by $T_2$.}\label{figure}
\end{figure}

\begin{prop}\label{structure_of_F} 
For each fixed $2\leq k \leq k^\ast$, let $F_{k-1}$ and $T_{\mathrm{rest}}$ be defined as above. We can order the vertex set $F_{k-1}$ as $F_{k-1}=\{w_1,\dots,w_\ell\}$ and obtain the following decomposition
\begin{equation}
H=\left(\bigcup_{i=1}^{\ell}T_i\right)\cup T_{\mathrm{rest}},
\qquad
\text{where $T_1,\ldots,T_\ell$ and $T_{\mathrm{rest}}$ are trees,}
\end{equation}
such that the following properties hold: 
\begin{enumerate}[(a)]
\item For each $i=1,\dots,\ell$, the vertex $w_i$ belongs to $V(T_i)$ and to no other $V(T_j)$.
\item For every $t<\ell$, the graphs
$T_{t+1}\setminus\{u\}$ and
\(H_t:=\bigcup_{i=1}^{t}T_i\)
intersect in at most one vertex.
\end{enumerate}
\end{prop}

\noindent\textbf{Remark.}
Let us sketch an outline of the proof. For convenience, we first assume that $H\backslash{u}$ is connected. The general case follows from the same construction, with only routine modifications needed to keep track of the different components of $H\backslash{u}$. Suppose that $F_{k-1}$ contains $\ell$ vertices. We decompose $H$ into the trees
$T_1,\ldots,T_\ell$ and $T_{\mathrm{rest}}$, where
$T_{\mathrm{rest}}$ may be the trivial tree consisting only of $u$,
such that each vertex in $F_{k-1}$ belongs to exactly one of
$T_1,\ldots,T_\ell$. The ordering of the vertices of $F_{k-1}$ corresponds to the order in which these trees are constructed.

We begin by selecting the first tree $T_1$; see the proof for the precise construction. Let $w_1$ be the unique vertex in $V(T_1)\cap F_{k-1}$. We then proceed iteratively: at each step, we select a new tree and identify a new vertex $w\in F_{k-1}$, continuing until all $\ell$ trees have been constructed. Moreover, each new tree is chosen so that its intersection with the union of the previously selected trees consists either of a single vertex, which may or may not be $u$, or of two vertices, one of which is $u$.

In other words, each subsequent tree may grow from a vertex different from $u$ in the existing structure, without intersecting $u$. When this is not possible, it may instead grow from $u$. Consequently, under the assumption that $H\backslash{u}$ is connected, the union of the previously selected trees, with $u$ removed, remains connected throughout the process. At each step, the new tree contains a previously uncovered vertex of $F_{k-1}$. This decomposition is motivated by the embedding step, in which we embed $H$ into $G$ by embedding the trees $T_1,\ldots,T_\ell$ first and then
$T_{\mathrm{rest}}$.

\begin{proof}[Proof of Proposition~\ref{structure_of_F}]
As mentioned earlier, we assume that $H$ has at least one cycle and that $H\backslash\{u\}$ is connected. In the general case, the same argument is repeated for the relevant components; the components not intersecting $F_{k-1}$, together with the vertex $u$, form $T_{\mathrm{rest}}$ as defined above.

We begin by selecting an arbitrary  vertex from $F_{k-1}$ and denote it by $w_1$. Let $T^-_1$ be the unique connected component of 
$$(H\backslash  \{u\} ) \backslash (F_{k-1}\backslash \{w_1\})$$ 
that contains $w_1$. 
Define $T_1$ to be the union of $T_1^-$ with the unique edge joining $T^-_1$ to the vertex $u$.
Let $H_1=T_1$. 
In Figure~\ref{figure}, the subgraph with red dotted edges illustrates an example of $T_1$.

Since $k\leq k^*$, we have $|F_{k-1}|\geq 2$. Since $H\backslash\{u\}$ is connected, the two sets
\[
V(H_1)\backslash\{u\}
\quad\text{and}\quad
F_{k-1}\backslash\{w_1\}
\]
are joined by at least one edge in $H\backslash\{u\}$. Choose such an edge and denote it by $pw_2$, where $p\in V(H_1)\backslash\{u\}$ and $w_2\in F_{k-1}\backslash\{w_1\}$. Note that $w_2$ has no other neighbor in $T_1^-$; otherwise, two distinct edges from $w_2$ to $T_1^-$, together with the path between their endpoints inside the tree $T_1^-$, would form a cycle in $H\backslash\{u\}$. We define $T_2^-$ to be the connected component of $$(H\backslash H_1) \backslash (F_{k-1}\backslash\{w_2\})$$ which contains $w_2$. Now there are two cases: 
\begin{enumerate}
\item 
If $p\in S_{k-1}^-\backslash\{u\}$, then $T_2^-$ is connected to $T_1$ only through the edge $pw_2$. In this case, we define $T_2$ as the union of $T_2^-$ and the edge $pw_2$. 
This case is illustrated by the blue dashed tree in Figure~\ref{figure}. 
\item If $p\in R_{k-1}$, 
let $P_{w_2u}$ be the unique $(k-1)$-path 
connecting $w_2$ to $u$. This path  intersects $H_1$ only at $u$.
Note that $T_2^-$ already contains $P_{w_2u}\backslash \{u\}$. 
In this case, $T_2$ is defined as the union of $T_2^-$ and two additional edges:   the edge $pw_2$ and the terminal edge of $P_{w_2u}$ incident to $u$. 
This case is illustrated by the green dashed tree in Figure~\ref{figure}; in the figure, this tree is denoted by $T_2$.
\end{enumerate}

In either of the two cases above,
we have obtained the tree $T_2$. Let $H_2=T_1\cup T_2$. By construction,  both conditions (a) and (b) are satisfied so far.

Now, proceed inductively. Suppose we have already located 
$w_1,\dots,w_t\in F_{k-1}$ 
and defined $H_t=\bigcup_{i=1}^t T_i$ for some $t<\ell$ satisfying conditions (a) and (b). 
Choose a vertex from $F_{k-1}\backslash\{w_1,\dots,w_t\}$ that has a neighbor in $H_t\setminus\{u\}$, and call it $w_{t+1}$. Then let $p$ be the unique vertex in $H_t\setminus \{u\}$ adjacent to $w_{t+1}$.
Define $T_{t+1}^-$ to be the connected component of $$(H\backslash H_t) \backslash (F_{k-1} \backslash \{w_{t+1}\})$$ that contains $w_{t+1}$. 
We have two cases: 
\begin{enumerate}
\item If $p\in S_{k-1}^- \backslash \{u\}$, define $T_{t+1}$ to be the union of $T_{t+1}^-$ and the edge $pw_{t+1}$.

\item If $p\in R_{k-1}$, then we have a 
unique $(k-1)$-path $P_{w_{t+1}u}$ connecting $w_{t+1}$ to $u$, and it intersects $H_t$ only at $u$.
Note that $T_{t+1}^-$ contains $P_{w_{t+1}u}\backslash \{u\}$. 
In this case, define 
$T_{t+1}$ as the union of 
$T_{t+1}^-$ and two edges:  the edge $pw_{t+1}$, and the terminal edge of $P_{w_{t+1}u}$ incident to $u$.

\end{enumerate}

For condition (a), each $w_i$ belongs to $V(T_i)$ and to no other $V(T_j)$, for $1\leq i,j\leq t+1$ with $i\neq j$.
For condition $(b)$, suppose $T_{t+1}$ intersects $H_{t}\backslash\{u\}$ at two vertices $v_1$ and $v_2$.
Then we can show that $(H_{t}\cup T_{t+1})\setminus \{u\}$ contains a cycle. Indeed, $H_{t}\setminus \{u\}$ is connected, so it contains a path between $v_1$ and $v_2$. Similarly,  $T_{t+1}\setminus \{u\}$ contains  another path between $v_1$ and $v_2$. Thus, there are two distinct paths between $v_1$ and $v_2$, forming a cycle. 

The above process eventually terminates when $T_\ell$ has been constructed and all vertices of $F_{k-1}$ have been covered. The vertices not covered by $T_1,\ldots,T_\ell$, together with the vertex $u$, form the graph $T_{\mathrm{rest}}$ defined above.
\end{proof}

\subsection{Key Embedding Proposition}
Hereafter, we suppose that $G$ is a bipartite graph with a designated root $\{r\}=L_0$. We define layer $L_i$ to be 
the set of vertices at distance $i$ from $r$; that is,  $V(G)=\{r\}\cup(\bigcup\{L_i\})$.
For any $U\subseteq L_i$ and $V\subseteq L_j$ with $i<j$, we  
denote by $\mathcal P(U,V)$ the set
of paths of the form $v_iv_{i+1}\cdots v_j$, where each $v_\ell\in L_\ell$ for $\ell=i,i+1, \dots, j$.
When one of the sets-- say $U$-- contains only a single vertex $u$, we write $\mathcal{P}(u, V)$ instead of $\mathcal{P}(\{u\}, V)$. For any vertex $v\in L_i$, we say that the neighbors of $v$ in $L_{i+1}$ are the {\it children} of $v$, and the neighbors of $v$ in $L_{i-1}$ are the {\it parents} of $v$. 

For any positive number $K$, we define
\begin{equation}
	\mathcal R_m(K) = \frac{K^m}{m+1} \binom{2m}{m} = K^m C_m, \quad \text{for all } m \geq 0,
\end{equation}
where $C_m = \frac{1}{m+1} \binom{2m}{m}$ is the $m$-th Catalan number. The Catalan numbers can also be defined recursively by setting $C_0 = 1$ and
\begin{equation}\label{catalan_recurrence}
	C_{m}=\sum_{i=1}^{m} C_{i-1}C_{m-i}, \text{ for all } m\geq 1.
\end{equation}
Catalan numbers arise in numerous combinatorial contexts. 
In our argument, however, their role is only to provide a convenient increasing sequence of constants satisfying the recurrence above. 
Indeed, we restrict ourselves to the above formula; for further details, see~\cite{stanley2015catalan}.

We now present the key embedding proposition, which gives sufficient
conditions, in terms of path counts between layers, for
embedding $H$. The proposition extends the embedding argument of
Lemma~4.2 of Bukh and Tait~\cite{Bukh_theta} to an arbitrary prescribed
family of \emph{layered paths}. Its proof combines this extension with
the structural decomposition established in
Proposition~\ref{structure_of_F}.

By a \emph{layered path} with respect to the layers
$L_0,L_1,\ldots,L_k$, we mean a path
$v_i v_{i+1}\cdots v_j$, for some $0\leq i<j\leq k$, such that
$v_t\in L_t$ for every $t=i,i+1,\ldots,j$.

\begin{prop}\label{Key_Embedding}
Let $G$ be a bipartite graph with maximum degree at most $\Delta d$.
Choose a root vertex $r\in V(G)$, and let
\[
L_0=\{r\},L_1,\ldots,L_k
\]
be the corresponding distance layers, where $2\leq k\leq k^\ast$.
Let $K\geq 4v(H)$ be a sufficiently large constant.

Let $\mathcal F$ be a nonempty family of layered paths of the form
\[
v_0(=r)v_1\cdots v_k,
\qquad v_i\in L_i.
\]
Let $A\subseteq L_1$ be the set of vertices occurring as the first vertex after the root in the paths of $\mathcal F$, and let $B\subseteq L_k$ be the set of their
terminal vertices.

For a path $P=v_0v_1\cdots v_k\in\mathcal F$ and integers
$0\leq i<j\leq k$, write
\[
P[i,j]:=v_iv_{i+1}\cdots v_j.
\]
For $x\in L_i$ and $y\in L_j$, define the path family
\begin{equation}\label{descending_family}
\mathcal F(x,y):=
\{P[i,j]: P=v_0v_1\cdots v_k\in\mathcal F,\ v_i=x,\ v_j=y\}.
\end{equation}

Assume that the following conditions hold:
\begin{enumerate}
    \item\label{PF_Path_Bound} Path bound between intermediate layers. For every $x\in L_i$ and $y\in L_j$, with
    \[
    0\leq i<j\leq k
    \qquad\text{and}\qquad
    (i,j)\neq(0,k),
    \]
    we have
    \[
    |\mathcal F(x,y)|
    \leq
    \mathcal R_{j-i-1}(K).
    \]

    \item\label{PF_First_Density} Large number of paths from $L_1$ to $L_k$ per vertex in $L_1$. 
    \[
    \frac{|\mathcal F|}{|A|}
    >
    Kk(\Delta d)^{k-2}.
    \]

\item\label{PF_Terminal_Density} Large number of paths from $L_1$ to $L_k$ per vertex in $L_k$. 
    \[
    \frac{|\mathcal F|}{|B|}
    >
    \mathcal R_{k-1}(K).
    \]
\end{enumerate}
Then $H$ can be embedded into the support graph $G_{\mathcal F}$ of
$\mathcal F$, and hence into $G$. Here, $G_{\mathcal F}$ is the
subgraph of $G$ consisting of all vertices and edges occurring in
members of $\mathcal F$.
\end{prop}

\begin{proof}
Starting with $\mathcal F_0:=\mathcal F$, we repeatedly apply the
following two cleaning operations whenever they are applicable, in an
arbitrary order, until neither operation can be applied. A vertex or an
edge is called \emph{active} if it occurs in at least one path of the
current family. For an active vertex $x\in L_j$, an \emph{active child}
of $x$ is a vertex $y\in L_{j+1}$ such that the edge $xy$ occurs in at
least one path of the current family.

\begin{enumerate}
    \item[(i)] If an active terminal vertex $b\in B$ is the terminal
    vertex of at most
    \[
    \frac{\mathcal R_{k-1}(K)}{2}
    \]
    current paths, delete all current paths ending at $b$.

    \item[(ii)] If an active vertex $x\in L_j$, where
    $1\leq j\leq k-1$, has fewer than $K/2$ active children in
    $L_{j+1}$, delete all current paths containing $x$.
\end{enumerate}

Each application deletes at least one current path. Since $\mathcal F$
is finite, the process terminates. We next show that some paths remain
after the process terminates. The estimates below hold regardless of
the order in which the two operations are performed.

Each terminal vertex can be selected in operation~(i) at most once,
since afterward all current paths ending at that vertex have been
deleted and it becomes inactive. Therefore, the total number of paths
deleted by operation~(i) is at most
\[
|B|\frac{\mathcal R_{k-1}(K)}{2}
<
\frac{|\mathcal F|}{2},
\]
by condition~\ref{PF_Terminal_Density}.

Now fix $j\in\{1,\ldots,k-1\}$. At the moment when an active vertex
$x\in L_j$ is deleted by operation~(ii), every current path containing
$x$ is determined by a layered prefix from a vertex of $A$ to $x$, one
of fewer than $K/2$ active edges from $x$ to $L_{j+1}$, and a
continuation from $L_{j+1}$ to $L_k$.

There are at most
\[
|A|(\Delta d)^{j-1}
\]
layered paths from $A$ to $L_j$. Once the edge from $L_j$ to
$L_{j+1}$ has been chosen, there are at most
\[
(\Delta d)^{k-j-1}
\]
possible continuations to $L_k$. Hence the total number of paths
deleted by operation~(ii) at layer $L_j$ is at most
\[
|A|(\Delta d)^{j-1}
\frac{K}{2}
(\Delta d)^{k-j-1}
=
\frac{K}{2}|A|(\Delta d)^{k-2}.
\]
Here every path is counted only when it is deleted.

Summing over $j=1,\ldots,k-1$, the total number of paths deleted by
operation~(ii) is at most
\[
\frac{(k-1)K}{2}|A|(\Delta d)^{k-2}
<
\frac{k-1}{2k}|\mathcal F|,
\]
where condition~\ref{PF_First_Density} was used.

Consequently, the total number of paths deleted by the two operations
is strictly less than
\[
\frac{|\mathcal F|}{2}
+
\frac{k-1}{2k}|\mathcal F|
<
|\mathcal F|.
\]
Thus some paths remain. Let $\mathcal F^\ast$ denote the nonempty
family remaining when the process terminates.

By construction, the surviving family $\mathcal F^\ast$ satisfies the following properties:
\begin{enumerate}
    \item[$(\bullet)$] Every active vertex in
    $L_1,\ldots,L_{k-1}$ has at least $K/2$ active children in the
    next layer.

    \item[$(\bullet\bullet)$] Every active terminal vertex $b\in L_k$
    satisfies
    \[
    |\mathcal F^\ast(r,b)|
    >
    \frac{\mathcal R_{k-1}(K)}{2}.
    \]
\end{enumerate}

We next show that:\\

$(\bullet\bullet\bullet)$ Every active terminal vertex $b\in L_k$ is joined
to $r$ by at least $K/2$ pairwise internally vertex-disjoint paths
belonging to $\mathcal F^\ast$.

Fix
\[
Q=rv_1\cdots v_{k-1}b\in\mathcal F^\ast(r,b).
\]
For each $j\in\{1,\ldots,k-1\}$, every path in
$\mathcal F^\ast(r,b)$ containing $v_j$ is determined by its
$r$--$v_j$ prefix and its $v_j$--$b$ suffix. Therefore,
\[
\begin{aligned}
&\left|
\left\{
P\in\mathcal F^\ast(r,b):
P\text{ meets }Q\text{ at an internal vertex}
\right\}
\right| \\
&\qquad\leq
\sum_{j=1}^{k-1}
|\mathcal F^\ast(r,v_j)|
\,|\mathcal F^\ast(v_j,b)|.
\end{aligned}
\]
Since $\mathcal F^\ast\subseteq\mathcal F$, condition~\ref{PF_Path_Bound} gives
\[
\begin{aligned}
\sum_{j=1}^{k-1}
|\mathcal F^\ast(r,v_j)|
\,|\mathcal F^\ast(v_j,b)|
&\leq
\sum_{j=1}^{k-1}
\mathcal R_{j-1}(K)\mathcal R_{k-j-1}(K) \\
&=
K^{k-2}
\sum_{j=1}^{k-1}
\mathcal C_{j-1}\mathcal C_{k-j-1} \\
&=
K^{k-2}\mathcal C_{k-1} \\
&=
\frac{\mathcal R_{k-1}(K)}{K},
\end{aligned}
\]
where the penultimate equality follows from
\eqref{catalan_recurrence}.

Let $\mathcal Q$ be a maximal family of pairwise internally
vertex-disjoint paths in $\mathcal F^\ast(r,b)$. By maximality, every
path in $\mathcal F^\ast(r,b)$ meets some member of $\mathcal Q$ at
an internal vertex. Hence
\[
|\mathcal F^\ast(r,b)|
\leq
|\mathcal Q|
\frac{\mathcal R_{k-1}(K)}{K}.
\]
Since
\[
|\mathcal F^\ast(r,b)|
>
\frac{\mathcal R_{k-1}(K)}{2},
\]
we obtain
\[
|\mathcal Q|>\frac{K}{2}.
\]
Thus Property $(\bullet\bullet\bullet)$ is established. 
In particular,
these paths have distinct first vertices in $L_1$, so the root $r$
also has at least $K/2$ active children.

Let $G^\ast$ be the support graph of $\mathcal F^\ast$. 
We now embed $H$ into $G^\ast$. Consider the decomposition
\[
H=
\left(\bigcup_{i=1}^{\ell}T_i\right)
\cup T_{\mathrm{rest}}
\]
provided by Proposition~\ref{structure_of_F}. We embed
$T_1,\ldots,T_\ell$ successively and then embed
$T_{\mathrm{rest}}$.

We first embed $T_1$. Map the vertex $u$ to the root $r$. We embed
\[
T_1\cap S_{k-1}^{-}
\]
greedily into the first $k-2$ layers, respecting distance from $u$,
and place $w_1$ in $L_{k-1}$. This is possible because every active
vertex in the first $k-2$ layers has at least $K/2\geq 2v(H)$ active children.

We then embed the remaining vertices of $T_1$ into
\[
G^\ast[L_{k-1}\cup L_k].
\]
Every active vertex in $L_{k-1}$ has at least $K/2$ active children
in $L_k$, and every active vertex in $L_k$ has at least $K/2$ active
parents in $L_{k-1}$. Since fewer than $v(H)$ vertices are used
throughout the embedding and
\(
K/2\geq 2v(H),
\)
the remaining part of $T_1$ can be embedded greedily while avoiding
all previously used vertices.

Suppose inductively that
\[
H_t:=\bigcup_{i=1}^{t}T_i
\]
has already been embedded for some $t<\ell$, and consider
$T_{t+1}$. By property~(b) of
Proposition~\ref{structure_of_F}, the graphs
$T_{t+1}\setminus\{u\}$ and $H_t$ intersect in at most one vertex.

If
\[
H_t\cap T_{t+1}=\{u\},
\]
we embed $T_{t+1}$ in the same way as $T_1$.

Otherwise, let
\[
v\in H_t\cap T_{t+1},
\qquad v\neq u.
\]
We distinguish two cases.

\medskip
\noindent
\textbf{Case 1.} $v\in S_{k-1}^{-}\setminus\{u\}$.

By the construction in Proposition~\ref{structure_of_F}, the vertex
$v$ is the parent of $w_{t+1}$. Since $w_{t+1}\in F_{k-1}$, the
vertex $v$ is at distance $k-2$ from $u$, and its image lies in
$L_{k-2}$. The image of $v$ has at least $K/2$ active children in
$L_{k-1}$. We may therefore choose an unused active child as the image
of $w_{t+1}$.

We then embed the remaining vertices of $T_{t+1}$ greedily into
\[
G^\ast[L_{k-1}\cup L_k],
\]
using active children when the previously embedded vertex lies in
$L_{k-1}$ and active parents when it lies in $L_k$. Since fewer than
$v(H)$ vertices are forbidden, the active degree condition guarantees
that the embedding succeeds.

\medskip
\noindent
\textbf{Case 2.} $v\in R_{k-1}$.

By the construction in Proposition~\ref{structure_of_F}, the vertex
$v$ is a child of $w_{t+1}$, and its image lies in $L_k$. Moreover,
the path $P_{w_{t+1}u}$ is contained in $T_{t+1}$.

The image of $v$ is joined to $r$ by at least $K/2$ pairwise
internally vertex-disjoint layered paths in $G^\ast$. Since these
paths are pairwise internally disjoint, every previously used vertex,
other than $r$ and the image of $v$, lies internally on at most one
of them. As fewer than $v(H)$ vertices have already been used and
\[
\frac{K}{2}\geq 2v(H),
\]
we may choose one of these paths whose internal vertices have not yet
been used.

We use the chosen path to embed
\[
P_{w_{t+1}u}\cup\{w_{t+1}v\},
\]
with $u$ mapped to $r$ and $v$ mapped to its already prescribed
image. In particular, $w_{t+1}$ is mapped to the penultimate vertex
of the chosen path in $L_{k-1}$.

We then embed the remaining vertices of $T_{t+1}$ greedily. Vertices
belonging to $S_{k-1}^{-}$ are embedded into the corresponding
distance layers using active children, and all other vertices are
embedded into
\[
G^\ast[L_{k-1}\cup L_k]
\]
using active children and active parents. Again, the active degree
condition and the inequality $K/2\geq 2v(H)$ ensure that all previously
used vertices can be avoided.

After embedding $T_1,\ldots,T_\ell$, we embed
$T_{\mathrm{rest}}$. Every vertex of
$T_{\mathrm{rest}}\setminus\{u\}$ has distance at most $k-2$ from
$u$, since no component defining $T_{\mathrm{rest}}$ intersects
$F_{k-1}$. Hence $T_{\mathrm{rest}}$ can be embedded greedily into
the first $k-2$ layers using the active-child property and avoiding
all previously used vertices.

Thus $H$ embeds into $G^\ast$, the support graph of
$\mathcal F^\ast$. Since $\mathcal F^\ast\subseteq\mathcal F$, the
graph $H$ also embeds into the support graph of $\mathcal F$, and
hence into $G$.
\end{proof}

We next extract a counting consequence of the path-family embedding
proposition. Fix a vertex $v$ in an intermediate layer and consider a
family $\mathcal F$ of layered paths from $v$ to $L_k$. Intuitively, a
terminal vertex $w$ is regarded as \textbf{bad} if there are too many paths in
$\mathcal F$ from $v$ to $w$. Assuming that all non-extreme subpath
counts satisfy the prescribed bounds, the following corollary shows
that the total number of paths ending at such bad vertices is small.

\begin{coro}
\label{Path_Family_Counting}
Let $G$ be an $H$-free bipartite graph with maximum degree at most
$\Delta d$. Let $r$ be a root, and let
\[
L_0=\{r\},L_1,\ldots,L_k
\]
be the corresponding distance layers in $G$, where $k\leq k^*$.
Let $K\geq 4v(H)$ be a sufficiently large constant. Fix an integer
$i\leq k-2$ and a vertex $v\in L_i$, and let $\mathcal F$ be a family
of layered paths from $v$ to $L_k$. Let $B\subseteq L_k$ be the set of
terminal vertices of the paths in $\mathcal F$.

Suppose that for every $x\in L_a$ and $y\in L_b$, with
\[
i\le a<b\le k
\qquad\text{and}\qquad
(a,b)\neq(i,k),
\]
we have
\[
|\mathcal F(x,y)|
\le
\mathcal R_{b-a-1}(K).
\]

Let
\[
W=\{w\in B:\ |\mathcal F(v,w)|>\mathcal R_{k-i-1}(K)\}.
\]
Then
\[
\sum_{w\in W}|\mathcal F(v,w)|
=
O(d^{k-i-1}).
\]
\end{coro}

\begin{proof}
If $W=\varnothing$, then the conclusion is immediate. Hence, assume
that $W\neq\varnothing$, and set $m:=k-i$.
Let
\[
\mathcal F_W:=\bigcup_{w\in W}\mathcal F(v,w).
\]
We regard $\mathcal F_W$ as a family of layered paths rooted at $v$,
with respect to the shifted layers
\[
L_0'=\{v\},
\qquad
L_j'=L_{i+j}\quad (1\leq j\leq m).
\]
Let $A\subseteq L_{i+1}$ be the set of first vertices occurring in
paths of $\mathcal F_W$. The terminal set of $\mathcal F_W$ is $W$.

Since $\mathcal F_W\subseteq\mathcal F$, the assumed bounds on
$\mathcal F(x,y)$ imply condition~\ref{PF_Path_Bound} of
Proposition~\ref{Key_Embedding} for every non-extreme pair in the
shifted layer system. Moreover, by the definition of $W$,
\[
\frac{|\mathcal F_W|}{|W|}
=
\frac{\sum_{w\in W}|\mathcal F(v,w)|}{|W|}
>
\mathcal R_{k-i-1}(K)
=
\mathcal R_{m-1}(K).
\]
Thus condition~\ref{PF_Terminal_Density} also holds.

Since $G$ is $H$-free, condition~\ref{PF_First_Density} cannot hold;
otherwise Proposition~\ref{Key_Embedding} would yield a copy of $H$.
Therefore,
\[
|\mathcal F_W|
\leq
|A|Km(\Delta d)^{m-2}.
\]
Since $A\subseteq N_G(v)$ and $\Delta(G)\leq\Delta d$, we have
\(
|A|\leq\Delta d.
\)
Consequently,
\[
\sum_{w\in W}|\mathcal F(v,w)|
=
|\mathcal F_W|
\leq
\Delta d\cdot Km(\Delta d)^{m-2}
=
O(d^{m-1})
=
O(d^{k-i-1}),
\]
as required.
\end{proof}

\subsection{Bad-set estimates}

We first note a gap on page~8 in the proof of Lemma~3.2
of~\cite{Bukh_theta}, at the displayed equality decomposing the paths
from $r$ to $B_i'$. There, the paths are classified according to a
vertex witnessing the badness of their terminal vertex. However, a
terminal vertex may be bad because of one vertex in $L_{i-1}$, while a
path reaching it passes through another vertex in that layer. Hence
the asserted counting equality need not hold.

Define the set of \emph{bad vertices} in $L_k$ by
\begin{equation}
\label{B'_definition}
B=
\left\{
v_k\in L_k:
\text{ there exist $0\leq i<k$ and $v_i\in L_i$ such that }
|\mathcal P(v_i,v_k)|>\mathcal R_{k-i-1}(K)
\right\}.
\end{equation}

Accordingly, for $v\in L_{i}$ and $v_k\in B$, we say that $v$ is a
\emph{witness} for $v_k$ if
\[
|\mathcal P(v,v_k)|>\mathcal R_{k-i-1}(K).
\]

To bound both the number of bad vertices and the number of paths from
the root to these vertices, we first partition the bad vertices
according to the largest index at which their badness is witnessed.
We then construct a descending residual-family decomposition of
$\mathcal P(r,B)$ and apply Corollary~\ref{Path_Family_Counting} at
each stage.

\begin{lemma}
\label{Induction_Step}
Let $2\leq k\leq k^*$, and let $G$ be a bipartite graph with root
$r$ and layers
\[
L_0=\{r\},L_1,\ldots,L_k.
\]
Assume that $G$ is $H$-free and has maximum degree at most
$\Delta d$. Let $K\geq 4v(H)$ be a sufficiently large constant.
Suppose that, for any $v_i\in L_i$ and $v_j\in L_j$ with
$0\leq i<j\leq k-1$,
\begin{equation}
\label{the_regular_condition}
|\mathcal P(v_i,v_j)|
\leq
\mathcal R_{j-i-1}(K).
\end{equation}
Then
\[
|\mathcal P(r,B)|=O(d^{k-1})
\qquad\text{and}\qquad
|B|=O(d^{k-1}).
\]
\end{lemma}

\begin{proof}
We first decompose $B$ according to the largest index $i$ for which a
vertex of $B$ is bad from some vertex in $L_{i-1}$. More precisely, for $i=k-1,\ldots,1$, define
\[
B_i''
=
\left\{
v_k\in L_k:
\text{ there exists }w\in L_{i-1}
\text{ such that }
|\mathcal P(w,v_k)|>\mathcal R_{k-i}(K)
\right\}.
\]
Define inductively
\[
B_{k-1}'=B_{k-1}'',
\qquad
B_i'=B_i''\setminus\bigcup_{t=i+1}^{k-1}B_t',
\quad i=k-2,\ldots,1.
\]

Since $\mathcal R_0(K)=1$ and there is at most one layered path between
two vertices in consecutive layers, no vertex of $B$ can satisfy the
defining inequality with a vertex in $L_{k-1}$. Hence every vertex of
$B$ belongs to $B_i''$ for some $i\in\{1,\ldots,k-1\}$.
The sets $B_i'$ are pairwise disjoint and
\[
B=\bigcup_{i=1}^{k-1}B_i'.
\]

Accordingly, let
\[
B_i''(v)
=
\left\{
v_k\in L_k:
|\mathcal P(v,v_k)|>\mathcal R_{k-i}(K)
\right\},
\]
and set
\[
B_i'(v):=B_i'\cap B_i''(v).
\]
Thus, $B_i'(v)$ consists of the vertices in $B_i'$ for which $v$ is a
witness. Clearly,
\[
B_i'=\bigcup_{v\in L_{i-1}}B_i'(v).
\]
This union need not be disjoint, since a vertex of $B_i'$ may have
several witnesses in $L_{i-1}$.

Fix $i\in\{1,\ldots,k-1\}$ and $v\in L_{i-1}$, and define
\[
\mathcal F_v
=
\bigcup_{w\in B_i'(v)}\mathcal P(v,w).
\]
We regard $\mathcal F_v$ as a family of layered paths from $v$ to
$L_k$. We verify the path bounds required by
Corollary~\ref{Path_Family_Counting}. Let $x\in L_a$ and $y\in L_b$,
where
\[
i-1\leq a<b\leq k
\qquad\text{and}\qquad
(a,b)\neq(i-1,k).
\]
If $b\leq k-1$, then \eqref{the_regular_condition} gives
\[
|\mathcal F_v(x,y)|
\leq
|\mathcal P(x,y)|
\leq
\mathcal R_{b-a-1}(K).
\]
Now suppose that $b=k$. Then $y\in B_i'(v)\subseteq B_i'$. If
$a\leq k-2$, the maximality of $i$ in the definition of $B_i'$ implies
\[
|\mathcal P(x,y)|
\leq
\mathcal R_{k-a-1}(K);
\]
otherwise $y$ would belong to $B_{a+1}''$ and hence to some
$B_t'$ with $t>i$. The case $a=k-1$ follows from
$\mathcal R_0(K)=1$, since there is at most one layered path between
adjacent vertices. Thus, in every case,
\[
|\mathcal F_v(x,y)|
\leq
\mathcal R_{b-a-1}(K).
\]

Moreover, for every $w\in B_i'(v)$,
\[
|\mathcal F_v(v,w)|
=
|\mathcal P(v,w)|
>
\mathcal R_{k-i}(K).
\]
Hence every terminal vertex of $\mathcal F_v$ belongs to the set $W$
in Corollary~\ref{Path_Family_Counting}. Therefore,
\[
\sum_{w\in B_i'(v)}|\mathcal P(v,w)|
=
O(d^{k-i}).
\]
Since
\[
\mathcal R_{k-i}(K)|B_i'(v)|
<
\sum_{w\in B_i'(v)}|\mathcal P(v,w)|,
\]
we obtain
\[
|B_i'(v)|=O(d^{k-i}).
\]

Since $|L_{i-1}|=O(d^{i-1})$, we obtain
\[
|B_i'|
\le
\sum_{v\in L_{i-1}}|B_i'(v)|
=
O(d^{k-1}).
\]
Consequently,
\[
|B|=O(d^{k-1}).
\]

It remains to estimate $|\mathcal P(r,B)|$. 
Starting from $\mathcal P(r,B)$, we successively delete suitable subfamilies while processing the layers from $L_{k-2}$ toward the root.
At stage $i$, we remove all remaining paths for which
the suffix from their vertex in $L_i$ to their terminal vertex occurs
too many times in the current residual family.

Set
\[
\mathcal P_{k-1}(r,B):=\mathcal P(r,B).
\]
We then define the residual families recursively. For
$i=k-2,k-3,\ldots,0$, suppose we have already defined 
$\mathcal P_{i+1}(r,B)$. Following the notation introduced in~\ref{descending_family}, for  $v\in L_i$ and $w\in B$,
 define 
the residual suffix family
\[
\mathcal P_{i+1}(v,w)
=
\left\{
Q:
\begin{array}{l}
\text{there exists }P\in\mathcal P_{i+1}(r,B)\\
\text{whose suffix from $v$ to $w$ is $Q$}
\end{array}
\right\}.
\]

For $v\in L_i$, define
\[
W_i(v)
=
\left\{
w\in B:
|\mathcal P_{i+1}(v,w)|
>
\mathcal R_{k-i-1}(K)
\right\},
\]
and let
\[
V_i
=
\left\{
v\in L_i:W_i(v)\neq\varnothing
\right\}.
\]

Let $\mathcal Q_i(r,B)$ be the set of paths
\(
P=rv_1\cdots v_i\cdots w
\)
in $\mathcal P_{i+1}(r,B)$ such that
\(
w\in W_i(v_i),
\)
and set
\[
\mathcal P_i(r,B)
=
\mathcal P_{i+1}(r,B)\setminus\mathcal Q_i(r,B).
\]

We {\bf claim} that
\(
|\mathcal Q_i(r,B)|=O(d^{k-1})
\), 
for every $i=0,\ldots,k-2$. 

\begin{proof}[Proof of the Claim.]
Fix such an $i$ and a vertex
$v\in V_i$, and consider the residual path family
\[
\mathcal F_v
=
\bigcup_{w\in W_i(v)}\mathcal P_{i+1}(v,w).
\]
We verify the path bounds required by
Corollary~\ref{Path_Family_Counting}. Let $x\in L_a$ and $y\in L_b$,
where
\[
i\leq a<b\leq k
\qquad\text{and}\qquad
(a,b)\neq(i,k).
\]
If $b\leq k-1$, then \eqref{the_regular_condition} gives
\[
|\mathcal F_v(x,y)|
\leq
|\mathcal P(x,y)|
\leq
\mathcal R_{b-a-1}(K).
\]
Suppose now that $b=k$. Then $a>i$. If $a\leq k-2$, stage $a$ was
processed before stage $i$, and the residual family
$\mathcal P_{i+1}(r,B)$ is contained in $\mathcal P_a(r,B)$.
Consequently,
\[
|\mathcal F_v(x,y)|
\leq
\mathcal R_{k-a-1}(K).
\]
If $a=k-1$, the same bound follows from
$\mathcal R_0(K)=1$, since there is at most one layered path between
adjacent vertices. Thus all the required non-extreme path bounds hold.

Moreover, by the definition of $W_i(v)$,
\[
|\mathcal P_{i+1}(v,w)|
>
\mathcal R_{k-i-1}(K)
\]
for every $w\in W_i(v)$. Therefore,
Corollary~\ref{Path_Family_Counting} gives
\[
\sum_{w\in W_i(v)}
|\mathcal P_{i+1}(v,w)|
=
O(d^{k-i-1}).
\]

Every path in $\mathcal Q_i(r,B)$ has a unique vertex $v$ in $L_i$
and a unique endpoint $w\in B$. By the definition of
$\mathcal Q_i(r,B)$, we then have $v\in V_i$ and $w\in W_i(v)$.
Moreover, the $r$--$v$ prefix of the path belongs to
$\mathcal P(r,v)$, while its $v$--$w$ suffix belongs to
$\mathcal P_{i+1}(v,w)$.

Conversely, for every $v\in V_i$, every $w\in W_i(v)$, every path in
$\mathcal P(r,v)$, and every path in $\mathcal P_{i+1}(v,w)$, the
concatenation of the two paths belongs to $\mathcal Q_i(r,B)$. Indeed,
since the paths proceed through successive layers, the prefix and the
suffix intersect only at $v$, and the residual condition after $L_i$
is determined entirely by the $v$--$w$ suffix. Thus this correspondence
is bijective. Hence
\[
|\mathcal Q_i(r,B)|
=
\sum_{v\in V_i}
|\mathcal P(r,v)|
\sum_{w\in W_i(v)}
|\mathcal P_{i+1}(v,w)|.
\]
Here, if $W_i(v)=\varnothing$, the corresponding inner sum is understood
to be zero.

It follows that
\[
|\mathcal Q_i(r,B)|
\leq
O(d^{k-i-1})
\sum_{v\in V_i}|\mathcal P(r,v)|
\leq
O(d^{k-i-1})|\mathcal P(r,L_i)|.
\]
Since the maximum degree of $G$ is at most $\Delta d$,
\[
|\mathcal P(r,L_i)|=O(d^i).
\]
Therefore,
\[
|\mathcal Q_i(r,B)|
=
O(d^i)O(d^{k-i-1})
=
O(d^{k-1}).
\]
\end{proof}

After all stages have been processed, the remaining paths form
$\mathcal P_0(r,B)$. By construction, for every $w\in B$ occurring
as the terminal vertex of a path in $\mathcal P_0(r,B)$,
\[
|\mathcal P_0(r,w)|
\leq
\mathcal R_{k-1}(K).
\]
Therefore,
\[
|\mathcal P_0(r,B)|
\leq
\mathcal R_{k-1}(K)|B|
=
O(d^{k-1}).
\]
Since the families
\[
\mathcal Q_{k-2}(r,B),\ldots,\mathcal Q_0(r,B),
\mathcal P_0(r,B)
\]
form a partition of $\mathcal P(r,B)$, we obtain
\[
|\mathcal P(r,B)|
=
\sum_{i=0}^{k-2}|\mathcal Q_i(r,B)|
+
|\mathcal P_0(r,B)|
=
O(d^{k-1}).
\]
This completes the proof.
\end{proof}

\subsection{Proof of Theorem~\ref{main}}
\label{subsec:proof_of_main}
The proof is based on the {\it graph exploration process}
of~\cite{Bukh_theta}, which we generalize to recursively restricted
families of layered paths. Throughout the proof, at each stage we delete
certain paths from the current family, but we do not delete vertices from
the ambient graph $G$. All subsequent path counts refer to the surviving
path family, or equivalently to its support graph, rather than to the
family of all layered paths in the original distance layers.

\begin{proof}[Proof of Theorem~\ref{main}]
Suppose for contradiction that
\(
e(G)=\omega(n^{1+1/k^*}).
\)
By Lemma~\ref{Kalmost_regular}, we may assume that $G$ is a balanced bipartite $\Delta$-almost regular graph on $n$ vertices with every vertex degree between $d$ and $\Delta d$. Since
\(
e(G)=\omega(n^{1+1/k^*}),
\)
it follows that
\(
d=\omega(n^{1/k^*}).
\)
Suppose that $G$ is $H$-free. We will eventually show that
\[
d=O(n^{1/k^*}),
\]
yielding a contradiction.

In the induction below, we begin with a family $\mathcal F$ of layered
paths from $r$ to $L_{k^\ast}$ and recursively restrict this family.
Let $\mathcal F_k$ denote the current path family at the $k$-th stage,
with $\mathcal F_1=\mathcal F$. For vertices $x$ and $y$, let
$\mathcal P_k(x,y)$ denote the family of $x$--$y$ subpaths occurring in
members of $\mathcal F_k$.

At stage $k$, we will define 
$B_k\subseteq L_k$, the set of bad vertices. Then we 
remove from $\mathcal F_k$ every path whose
vertex in $L_k$ belongs to $B_k$, and denote the remaining family by
$\mathcal F_{k+1}$. The ambient graph $G$ and its original distance
layers remain unchanged. The active graph after the $k$-th cleaning is
the support graph of $\mathcal F_{k+1}$, namely the subgraph of $G$
consisting of all vertices and edges occurring in its members. 
We organize the induction step through the following
claim.

We {\bf claim} that, for each $1\leq k\leq k^\ast$, we can recursively
define a set $B_k\subseteq L_k$ such that
\begin{align}
&|B_k|=O(d^{k-1})
\quad\text{and}\quad
|\mathcal P_{k-1}(r,B_k)|=O(d^{k-1}),
\label{Bk_size_tau}\\
&|\mathcal P_k(v_i,v_j)|
\leq \mathcal R_{j-i-1}(K)
\quad\text{for all }\
v_i\in L_i,\ v_j\in L_j,\ 0\leq i<j\leq k,
\label{regular_condition}\\
&|\mathcal P_k(r,L_k)|=\Omega(d^k).
\label{linear_path_number_estimate}
\end{align}

\begin{proof}[Proof of the Claim]

We prove the claim by induction on $k$. Recall that
$\mathcal F_0=\mathcal F$. For $k=1$, set $B_1=\varnothing$, so that
$\mathcal F_1=\mathcal F_0$. Then
\[
|B_1|=0,
\qquad
|\mathcal P_0(r,B_1)|=0,
\]
and
\[
|\mathcal P_1(r,L_1)|\geq d.
\]
Moreover, since there is at most one layered path between two vertices
in consecutive layers,~\eqref{regular_condition} also holds for $k=1$.
Thus the claim holds in the base case.

Assume that the conclusion holds for all $k\leq \ell$, where
$\ell\leq k^\ast-1$.

Define the set $B'\subseteq L_{\ell+1}$ by
\begin{equation}\label{B'_definition2}
B'
=
\left\{
v\in L_{\ell+1}:
\begin{array}{l}
\text{there exist }0\leq i<\ell+1
\text{ and }v_i\in L_i\\
\text{such that }
|\mathcal P_\ell(v_i,v)|
>
\mathcal R_{\ell-i}(K)
\end{array}
\right\}.
\end{equation}

Next, let $B''\subseteq L_{\ell+1}\setminus B'$ be the set of vertices
having at least
\(
\Delta\mathcal R_\ell(K)
\)
neighbors in $B_\ell$. Since every vertex of $B_\ell$ has degree at
most $\Delta d$, we have
\[
\Delta\mathcal R_\ell(K)|B''|
\leq
\Delta d|B_\ell|.
\]
By~\eqref{Bk_size_tau}, it follows that
\begin{equation}\label{B''_estimate}
|B''|
\leq
\frac{d|B_\ell|}{\mathcal R_\ell(K)}
=
O(d^\ell).
\end{equation}

Now set
\begin{equation}\label{B_ell_plus_one_definition}
B_{\ell+1}=B'\cup B''.
\end{equation}
Define
\begin{equation}\label{F_ell_deleted_family}
\mathcal F_\ell[B_{\ell+1}]
:=
\left\{
P\in\mathcal F_\ell:
V(P)\cap B_{\ell+1}\neq\varnothing
\right\}.
\end{equation}
Since $B_{\ell+1}\subseteq L_{\ell+1}$, this is precisely the family
of paths in $\mathcal F_\ell$ whose vertex in $L_{\ell+1}$ belongs to
$B_{\ell+1}$. We then set
\begin{equation}\label{F_ell_plus_one_definition}
\mathcal F_{\ell+1}
=
\mathcal F_\ell\setminus
\mathcal F_\ell[B_{\ell+1}].
\end{equation}

Applying Lemma~\ref{Induction_Step} to the current family
$\mathcal F_\ell$ and using~\eqref{B''_estimate}, we obtain
\[
|B_{\ell+1}|
\leq
|B'|+|B''|
=
O(d^\ell)+O(d^\ell)
=
O(d^\ell).
\]

Again by Lemma~\ref{Induction_Step},
\[
|\mathcal{P}_{\ell+1}(r,B')|=O(d^{\ell}).
\]
Moreover, since $B''\subseteq L_{\ell+1}\setminus B'$, the root $r$
is not a witness for any vertex of $B''$. Hence
\[
|\mathcal P_{\ell+1}(r,B'')|
\leq \mathcal R_{\ell}(K)|B''|
=O(d^\ell).
\]
Therefore,~\eqref{Bk_size_tau} holds for $k=\ell+1$.

Since
\(
\mathcal F_{\ell+1}\subseteq \mathcal F_\ell,
\)
the induction hypothesis implies~\eqref{regular_condition} for all
$0\leq i<j\leq\ell$. Now let $v\in L_{\ell+1}$ occur in
$\mathcal F_{\ell+1}$. Then $v\notin B'$, and hence, by the definition
of $B'$,
\[
|\mathcal P_\ell(v_i,v)|
\leq
\mathcal R_{\ell-i}(K)
\]
for every $v_i\in L_i$ and $0\leq i<\ell+1$. Since
\[
\mathcal P_{\ell+1}(v_i,v)
\subseteq
\mathcal P_\ell(v_i,v),
\]
property~\eqref{regular_condition} holds for all
$0\leq i<j\leq\ell+1$.

Next, consider an active vertex $v\in L_\ell$, that is, a vertex
occurring in a path of $\mathcal F_\ell$. By
\eqref{regular_condition}, it has at most
$\mathcal R_{\ell-1}(K)$ active parents in $L_{\ell-1}$. Moreover,
since $v\notin B_\ell$, it has fewer than
$\Delta\mathcal R_{\ell-1}(K)$ neighbors in $B_{\ell-1}$. Therefore,
$v$ has at least
\(
d-(\Delta+1)\mathcal R_{\ell-1}(K)
\)
neighbors in $L_{\ell+1}$.

Consequently, before the $(\ell+1)$-th cleaning,
\begin{equation}\label{path_count_before_cleaning}
|\mathcal P_\ell(r,L_{\ell+1})|
\geq
\bigl(d-(\Delta+1)\mathcal R_{\ell-1}(K)\bigr)
|\mathcal P_\ell(r,L_\ell)|
=
\Omega(d^{\ell+1}).
\end{equation}
After removing the paths passing through $B_{\ell+1}$, we obtain
\begin{align}
|\mathcal P_{\ell+1}(r,L_{\ell+1})|
&=
|\mathcal P_\ell(r,L_{\ell+1})|
-
|\mathcal P_\ell(r,B_{\ell+1})| \notag\\
&=
\Omega(d^{\ell+1})-O(d^\ell)
=
\Omega(d^{\ell+1}).
\end{align}
Thus~\eqref{linear_path_number_estimate} holds for
$k=\ell+1$.
\end{proof}

We run the induction until $k=k^\ast$. By
\eqref{linear_path_number_estimate},
\[
|\mathcal P_{k^\ast}(r,L_{k^\ast})|
=
\Omega(d^{k^\ast}).
\]

On the other hand, by~\eqref{regular_condition}, for every
$v\in L_{k^\ast}$,
\[
|\mathcal P_{k^\ast}(r,v)|
\leq
\mathcal R_{k^\ast-1}(K).
\]
Therefore,
\[
|\mathcal P_{k^\ast}(r,L_{k^\ast})|
=
\sum_{v\in L_{k^\ast}}
|\mathcal P_{k^\ast}(r,v)|
\leq
\mathcal R_{k^\ast-1}(K)|L_{k^\ast}|
\leq
\mathcal R_{k^\ast-1}(K)n.
\]
It follows that
\(
d=O\bigl(n^{1/k^\ast}\bigr),
\)
which concludes the proof.
\end{proof}

\section{Generalized Cube}
The proof of Theorem~\ref{new_exponents} proceeds in two parts, establishing the upper and lower bound in the subsections below.   
\subsection{Upper Bounds} 
\begin{prop}\label{upper_generalized_cube}
 	For any integers $k\geq 2$ and $t\geq 2$,
	\begin{equation}
 		\text{ex}(n,E^+_{k,t})=
 			O(n^{\frac{3k-1}{2k-1}}).
 	\end{equation}
 \end{prop} 
\begin{proof} Suppose, for contradiction, that $G$ is an $n$-vertex graph with $e(G)=\omega(n^{\frac{3k-1}{2k-1}})$.
By Lemma~\ref{Kalmost_regular}, 
we may assume that $G$ is bipartite,  
$\Delta$-almost regular, and has minimum degree $\delta=\delta(G)=\omega(n^{\frac{k}{2k-1}})$. Fix a vertex $r$ and let $L_1$ and $L_2$ denote the sets of vertices at graph distance $1$ and $2$ from $r$, respectively. By the minimum degree assumption, $|L_1|=\Theta(\delta)$, and clearly $|L_2|< n$. The average degree of vertices in $L_2$ toward $L_1$ is
\[
d_2:=\frac{e(G[L_1\cup L_2])}{|L_2|}
=\Omega\left(\frac{\delta^2}{n}\right)
=\omega\left(n^{\frac{1}{2k-1}}\right).
\]

By Lemma~\ref{reduction}, we may further assume that within $G[L_1\cup L_2]$: 
\begin{enumerate}
    \item every vertex in $L_1$ has degree  $\Theta(\delta)$,
    \item every vertex in $L_2$ has degree  $\Omega(\delta^2/n)$, and
    \item $|L_1|=\Theta(\delta)$ remains unchanged. 
\end{enumerate}

Choose any vertex $r'\in L_1$ and let $$U_1=N_{L_2}(r'), \ |U_1|=\Theta(\delta). $$ 
Define $U_2$ as the set of vertices in $L_1$ that lie at distance $2$ from $r'$ inside $G[L_1 \cup L_2]$. Since $|U_2|\leq |L_1|=\Theta(\delta)$, both parts of the bipartite subgraph $G[U_1\cup U_2]$ have size $O(\delta)$. The number of edges in this subgraph satisfies  $$\Omega(\delta^3/n)=
\omega(\delta^{1+1/k}).$$ By the classical result that $\text{ex}(n,\theta_{k,t})=O(n^{1+1/k})$ (see~\cite{Faudree1983} and \cite{doi:10.1137/21M1408439}), a copy of $\theta_{k,t}$ must embed into $G[U_1\cup U_2]$. Finally, note that every vertex in $U_1$ is adjacent to $r'$, and every vertex in $U_2$ is adjacent to $r$. Together with the edge $rr'$, these attachments extend the $\theta_{k,t}$ found inside $G[U_1\cup U_2]$ to a copy of $E^+_{k,t}$ in $G$, contradicting the assumption that $G$ is $E^+_{k,t}$-free. 
\end{proof}

\subsection{Lower Bounds}
Suppose a graph $H$ contains an independent set of vertices 
$R=\{r_1, \dots,r_p\}$, 
which we call the \textit{set of roots}. 
Following~\cite{bukh2018rational}, we define the notion of the \textit{unrooted density}, which is 
the rational number obtained by dividing the number of edges $e(H)$ by 
the number of unrooted vertices. More precisely, we write $$\rho_r(H)=\frac{e(H)}{v(H)-|R|}.$$ 
If $S\subseteq V(H)\backslash R$, we define the \textit{local unrooted density} by $\rho_r(S)=e(S)/|S|$, where $e(S)$ is the number of edges incident to at least one vertex of $S$ in the host graph $H$. The $R$-rooted graph $H$ is called \textit{balanced} if, for every 
$S\subseteq V(H)\backslash R$,  the local unrooted density is at least the unrooted density, 
that is, 
$$\rho_r(S)\geq \rho_r(H).$$ 

In this setting, for any positive integer $t$, we define the \textit{$R$-rooted $t$-blow-up}
of $H$ as follows. Fix the root set $R$.
There are $t$ copies of $H$, which share the same root set $R$, and the $t$ sets of 
unrooted vertices are pairwise disjoint. We write $H_R^{(t)}$ for such a graph. 
Clearly, $\rho_r(H_R^{(t)})=\rho_r(H)$. More generally, we may also consider $t$ copies of $H$ with the same root set $R$, but allow the unrooted vertex sets of different copies to intersect arbitrarily. In this case, the union of all the copies may yield a graph $H'$ different from $H_R^{(t)}$.  
By Lemma 2.2 of~\cite{bukh2018rational}, for any such graph $H'$, we have 
\begin{equation}
e(H') \geq  (v(H')- |R|) \times \rho_r(H).
\end{equation}

Note that Lemma~2.2 of~\cite{bukh2018rational} is stated for blow-ups of
balanced rooted trees. The random-algebraic construction used there is
bipartite. The same proof applies to balanced rooted bipartite graphs:
the required density estimate remains valid,
for arbitrary unions of rooted copies. We therefore record the following
bipartite version of the result.

\begin{theorem}\label{rooted-lower-bound}
Let $H$ be a balanced rooted bipartite graph with independent root set
$R$. Suppose that $H$ has $e$ edges and $s$ unrooted vertices, so that
its unrooted density is
\[
\rho_r(H)=\frac{e}{s}.
\]
Then there exists a sufficiently large integer $t_0$ such that, for
every $t\geq t_0$,
\[
\operatorname{ex}\bigl(n,H_R^{(t)}\bigr)
=
\Omega\left(n^{2-\frac{1}{\rho_r(H)}}\right)
=
\Omega\left(n^{2-\frac{s}{e}}\right).
\]
\end{theorem}

Next we consider a family of graphs $F_{k,t}$ defined as follows. For any $k\geq 2$, let $P_{k-1}=v_1v_2\cdots v_{k-1}$ be a path with $k-1$ vertices. 
Consider the root set $R=\{r_1,r_2,r_3,r_4\}$, where the four root vertices are also re-named as follows, 
$r_1=v_0,r_2=v_k, r_3=x$, and $r_4=y$. Then we join new edges $v_0v_1$, $v_{k-1}v_k$, $xv_{2j}$ for all $1\leq j \leq \lfloor (k-1)/2 \rfloor$, and $yv_{2j-1}$ for all $1\leq j\leq \lfloor k/2 \rfloor$. 
Denote by $F_{k,1}$ the above graph, then we define the $R$-rooted $t$-blow up of the above graph as $F_{k,t}$. 

\begin{prop}\label{balanced} 
For any $k\geq 2$, the rooted graph $F_{k,1}$ is balanced.
\end{prop}
\begin{proof} Observe that 
the unrooted density of $F_{k,1}$ is $\rho_r=\frac{2k-1}{k-1}$, since $F_{k,1}$ 
has $k-1$ unrooted vertices and $2k-1$ edges. In fact, the edge number can be counted as follows.
For each of the unrooted vertices $v$, it has exactly one incident edge joining either $x$ or $y$, and this gives a total of $k-1$ edges. 
The path $P_{k-1}$, together with edges $v_0v_1$ and $v_{k-1}v_k$, contributes $k$ edges. So the total edge number is $k-1+k=2k-1$.

Now we consider any $S\subset F_{k,1}\backslash R$, with its size $|S|=s\leq k-1$. The set $S$ is seen as 
a vertex subset of the path $P_{k-1}$.
Then we want to compute the local unrooted density $\rho_r(S)$. We note that 
each vertex in $S$ has degree exactly $3$, which gives a total count $3s$. On the other hand, the choice of the subset $S$ of the path $P_{k-1}$ might disconnect the path $P_{k-1}$.
So in the above total count, the involved edge in the path $P_{k-1}$ 
might be counted once or twice, depending on whether or not both of its endpoints belong to $S$, and there are in total at most $s-1$ edges counted twice. 
Since $s\leq k-1$, we conclude that the total number of edges which are incident to $S$ is at least $3s -(s-1) = 2s+1$. So the local unrooted density $\rho_r(S)\geq \frac{2s+1}{s}\geq \frac{2k-1}{k-1}=\rho_r(F_{k,1})$. This shows that the rooted graph $F_{k,1}$ is balanced.
\end{proof}

\begin{lemma}
	The graph $F_{k,t}$ is a subgraph of 
	the graph $E^+_{k,t}$ defined in the introduction. 
	\end{lemma}

\begin{proof}In fact $F_{k,t}$ is contained in $E_{k,t}$, which is a subgraph of $E_{k,t}^+$. The proof follows by comparing definitions. 
Observe that the graph $E_{k,t}$ is exactly the graph $F_{k,t}$ with two additional edges. In the case that $k$ is odd, the two new edges are $xv_0$ and $yv_k$. In the case that $k$ is even, the two new edges are $xv_0$ and $xv_k$.
	\end{proof}

	\begin{proof}[Proof of Theorem~\ref{new_exponents}] Recall that for any positive integers $k$ and $t$, 
	the graph $F_{k,t}$ is the $R$-rooted $t$-blow-up of the graph $F_{k,1}$ with $\rho_r=\frac{2k-1}{k-1}$.
	We have also checked in Lemma~\ref{balanced} that $F_{k,1}$ is balanced. 
	So, applying Theorem~\ref{rooted-lower-bound}, for sufficiently large $t$, 
	 we have
		\begin{equation}
			\text{ex}(n,E^+_{k,t})\geq \text{ex}(n,F_{k,t})=
			\Omega(n^{\frac{3k-1}{2k-1}}).
		\end{equation}
		Combining with Proposition~\ref{upper_generalized_cube},  the conclusion follows. 
	\end{proof}

\section{Final Remark}

We now turn our attention to graphs whose minimum FVN is larger than $1$. 
Consider the theta graph $\theta_{3,t}$, viewed as a bipartite graph with vertex bipartition $A\cup B$. For each vertex in $A$, attach a path of length $2$ that ends at a common vertex $x$. Let $L_t$ denote the resulting graph. It was shown in~\cite{jiang2022turan} that, for sufficiently large $t$, $$\text{ex}(n,L_t)=\Theta(n^{7/5}).$$ Note that this graph has girth $6$. This example shows that when the minimum FVN of a bipartite graph is 2, the Turán exponent is no longer determined solely by the minimum cycle length. Consequently, it remains unclear which additional structural properties of a graph--beyond its girth--play a role in determining its Turán exponent.

This question is closely related to the well-known {\it Rational Exponent Conjecture}, usually attributed to Erd\H{o}s and Simonovits (see~\cite{erdHos1981combinatorial}). The conjecture asks whether every rational number $r\in[1,2]$ can arise as the Turán exponent of some graph $H$; that is, whether for each such $r$ there exists a graph $H$ satisfying $\text{ex}(n,H)=\Theta(n^r)$.
When this holds, we say that the rational number $r$ is realizable as a Tur\'an exponent.

The Rational Exponent Conjecture has seen substantial progress in recent years. In a breakthrough work, Bukh and Conlon~\cite{bukh2018rational} proved a family version of the conjecture. Since then, many new families of rational exponents have been realized.
 We refer the reader to ~\cite{jiang2022turan, kang2021rational, conlon2021more, janzer2020extremal, jiang2023many, conlon2022rational} and others for further developments in this direction. 

Motivated by the Rational Exponent Conjecture and by our main theorem, 
we propose the following conjecture concerning the (non-)realizability of Tur\'an exponents by graphs with bounded minimum FVN.

\begin{conj}\label{feedback-exponent} [The Feedback-Exponent Conjecture]\label{conj}
    Let $k_0\geq 1$ be an integer. Then there exists a rational number 
$p/q\in[1,2]$ that cannot be realized as the Turán exponent of any bipartite graph whose minimum FVN is at most $k_0$.
\end{conj}

The case $k_0=1$ of Conjecture~\ref{feedback-exponent} follows from Theorem~\ref{main}. Indeed,
let $H$ be a connected bipartite graph with minimum FVN one. If $H$
has girth four, then $C_4\subseteq H$, and hence
\[
\operatorname{ex}(n,H)
\geq \operatorname{ex}(n,C_4)
=\Omega(n^{3/2}).
\]
Together with Theorem~1.1, this gives
\[
\operatorname{ex}(n,H)=\Theta(n^{3/2}).
\]
If $H$ has girth at least six, then $k^\ast\geq 3$, and Theorem~1.1
gives
\[
\operatorname{ex}(n,H)=O(n^{4/3}).
\]
Thus no exponent in
\[
\left(\frac{4}{3},\frac{3}{2}\right)
\cup
\left(\frac{3}{2},2\right)
\]
can be realized by a connected bipartite graph with minimum FVN one.

The disconnected case follows as well. Since the feedback vertex number
is additive over connected components, such a graph has at most one
cyclic component, while all remaining components are forests. Fixed
forest components affect the extremal number by at most a linear term,
and hence do not change any Turan exponent larger than one.

As a next step, resolving Conjecture~\ref{feedback-exponent} for $k_0=2$ would provide
valuable insight into the structural features, beyond girth alone, that govern
the Turán exponent.

Finally, we also suspect that when attention is restricted to graphs with minimum FVN equal to $1$ or $2$, 
the set of realizable exponents has no nontrivial accumulation points other than $1$.

\appendix
\section{Bouquets of $8$-cycles}
For an integer $t\geq 1$, let $W_t$ be
the graph obtained by taking $t$
vertex-disjoint copies of $C_8$ and identifying one vertex from each cycle into a single common vertex.
Thus, any two of the $t$
cycles intersect exactly at this common vertex.

\begin{prop}
\label{prop:C8_bouquet}
For every fixed integer $t\geq 1$,
\[
\operatorname{ex}(n,W_t)
=
\Theta_t\bigl(\operatorname{ex}(n,C_8)\bigr).
\]
More precisely,
\[
\operatorname{ex}(n,C_8)
\leq
\operatorname{ex}(n,W_t)
\leq
2^{14(t-1)+2}\operatorname{ex}(n,C_8)
+
7(t-1)n.
\]
\end{prop}

\begin{proof}
Since $C_8$ is a subgraph of $W_t$, every $C_8$-free graph is also
$W_t$-free. Therefore,
\[
\operatorname{ex}(n,C_8)
\leq
\operatorname{ex}(n,W_t).
\]

We prove the second inequality. Let $G$ be a $W_t$-free graph on
$n$ vertices, and write
\(
m:=e(G).
\)

For each vertex $v\in V(G)$, choose a maximal family
$\mathcal C_v$ of copies of $C_8$ containing $v$ such that any two
distinct members of $\mathcal C_v$ intersect only at $v$. Since
$G$ is $W_t$-free,
\(
|\mathcal C_v|\leq t-1.
\)

Define
\[
S_v
:=
\bigcup_{C\in\mathcal C_v}
\bigl(V(C)\setminus\{v\}\bigr).
\]
The sets $V(C)\setminus\{v\}$, for $C\in\mathcal C_v$, are pairwise
disjoint and each has size $7$. Hence
\begin{equation}
\label{eq:hitting_set_size}
|S_v|
=
7|\mathcal C_v|
\leq
7(t-1).
\end{equation}

Moreover, every copy of $C_8$ containing $v$ also contains a vertex
of $S_v$. Indeed, otherwise it would intersect every member of
$\mathcal C_v$ only at $v$, and hence it could be added to
$\mathcal C_v$, contradicting the maximality of $\mathcal C_v$.

Independently assign to each vertex $x\in V(G)$ a random variable
$\xi_x\in\{0,1\}$ such that
\[
\mathbb P(\xi_x=0)
=
\mathbb P(\xi_x=1)
=
\frac12.
\]
Define the random set
\[
U
:=
\left\{
v\in V(G):
\xi_v=1
\text{ and }
\xi_x=0
\text{ for every }x\in S_v
\right\}.
\]

We claim that $G[U]$ is $C_8$-free. Suppose, to the contrary, that
$C$ is a copy of $C_8$ contained in $G[U]$. Choose any vertex
$v\in V(C)$. Since $S_v$ intersects every copy of $C_8$ containing
$v$, there exists
\[
x\in S_v\cap\bigl(V(C)\setminus\{v\}\bigr).
\]
Since $v\in U$, the definition of $U$ gives $\xi_x=0$. On the other
hand, $x\in U$ gives $\xi_x=1$, a contradiction. Hence $G[U]$ is
$C_8$-free.

Call an edge $uv\in E(G)$ \emph{exceptional} if
\(
u\in S_v\text{ or }
v\in S_u.
\)
By \eqref{eq:hitting_set_size}, the number of exceptional edges is
at most
\[
\sum_{v\in V(G)}|S_v|
\leq
7(t-1)n.
\]

Now let $uv$ be a non-exceptional edge. Then
\(
u\notin S_v\text{ and }
v\notin S_u.
\)
Thus, the conditions defining the events $u\in U$ and $v\in U$ do
not conflict. Using \eqref{eq:hitting_set_size}, we obtain
\[
\begin{aligned}
\mathbb P(u,v\in U)
&=
2^{-\left|\{u,v\}\cup S_u\cup S_v\right|}\\
&\geq
2^{-2-|S_u|-|S_v|}\\
&\geq
2^{-14(t-1)-2}.
\end{aligned}
\]
Therefore,
\[
\mathbb E\bigl[e(G[U])\bigr]
\geq
2^{-14(t-1)-2}
\bigl(m-7(t-1)n\bigr).
\]
Hence there exists a choice of $U$ such that
\[
e(G[U])
\geq
2^{-14(t-1)-2}
\bigl(m-7(t-1)n\bigr).
\]
Since $G[U]$ is $C_8$-free,
\[
e(G[U])
\leq
\operatorname{ex}(n,C_8).
\]
It follows that
\[
m
\leq
2^{14(t-1)+2}\operatorname{ex}(n,C_8)
+
7(t-1)n.
\]
Taking the maximum over all $W_t$-free graphs $G$ proves the stated
upper bound.

Finally, since
\[
\operatorname{ex}(n,C_8)\geq n-1,
\]
the additive linear term can be absorbed into
$\operatorname{ex}(n,C_8)$, with a constant depending only on $t$.
Therefore,
\[
\operatorname{ex}(n,W_t)
=
\Theta_t\bigl(\operatorname{ex}(n,C_8)\bigr).
\]
\end{proof}

Let $\Theta_{1,7,7}$ denote the generalized theta graph consisting of
two vertices joined by three internally vertex-disjoint paths of
lengths $1$, $7$, and $7$. Equivalently, $\Theta_{1,7,7}$ is obtained
from two disjoint copies of $C_8$ by identifying one edge of the first copy
with one edge of the second copy.

\begin{prop}\label{theta177}
For every positive integer \(n\),
\[
\operatorname{ex}(n,C_8)
\leq
\operatorname{ex}(n,\Theta_{1,7,7})
\leq
2^8\operatorname{ex}(n,C_8).
\]
Consequently,
\[
\operatorname{ex}(n,\Theta_{1,7,7})
=
\Theta\bigl(\operatorname{ex}(n,C_8)\bigr).
\]
\end{prop}

\begin{proof}
The lower bound follows from \(C_8\subseteq\Theta_{1,7,7}\).

For the upper bound, let \(G\) be an \(n\)-vertex
\(\Theta_{1,7,7}\)-free graph. For each edge \(e=uv\), if there is a
copy \(C_e\) of \(C_8\) containing \(e\), choose one and set
\[
S_e:=V(C_e)\setminus\{u,v\};
\]
otherwise, set \(S_e=\varnothing\). Thus \(|S_e|\leq 6\). Moreover,
every copy of \(C_8\) containing \(e\) meets \(S_e\), since otherwise
it and \(C_e\) would intersect exactly in \(e\), and their union would
form a copy of \(\Theta_{1,7,7}\).

Independently assign to each vertex \(x\in V(G)\) a random label
\(\xi_x\in\{0,1\}\), with both values having probability \(1/2\).
Define a random subgraph $G'$ as follows. Keep an edge \(e=uv\) in \(G'\) if
\[
\xi_u=\xi_v=1
\qquad\text{and}\qquad
\xi_x=0\quad\text{for every }x\in S_e.
\]

The resulting spanning subgraph \(G'\) is \(C_8\)-free. Indeed, if a
copy \(C\) of \(C_8\) survived, then for any edge \(e\in E(C)\), some
vertex \(x\in V(C)\cap S_e\) would satisfy \(\xi_x=0\), whereas the
survival of an edge of \(C\) incident with \(x\) would imply
\(\xi_x=1\), a contradiction.

Each edge is kept in $G'$ with probability at least
\[
2^{-2-|S_e|}\geq 2^{-8}.
\]
Hence, by linearity of expectation,
\[
\mathbb{E}[e(G')]\geq 2^{-8}e(G).
\]
Therefore, there exists a realization of $G'$ such that
\[
e(G')\geq 2^{-8}e(G).
\]
Since this realization of $G'$ is $C_8$-free, we have
\[
2^{-8}e(G)
\leq e(G')
\leq \operatorname{ex}(n,C_8).
\]
Consequently,
\[
e(G)\leq 2^8\operatorname{ex}(n,C_8).
\]
\end{proof}

\begin{Remark}
More generally, for every fixed \(t\geq 2\),
\[
\operatorname{ex}\!\left(
n,\Theta_{1,\underbrace{7,\ldots,7}_{t\text{ times}}}
\right)
=
\Theta_t\bigl(\operatorname{ex}(n,C_8)\bigr),
\]
where the implicit constants depend only on \(t\). The proof is a
straightforward variant of the proof of Proposition~\ref{theta177}.
\end{Remark}

\bibliographystyle{plain}
\addcontentsline{toc}{chapter}{Bibliography}
\bibliography{forest}
\end{document}